\documentclass[12pt]{extarticle}
\usepackage{amsmath, amsthm, amssymb, color,xcolor}
\usepackage[ pdfborderstyle={/S/U/W 1}]{hyperref}
\definecolor{MyBlue}{HTML}{210cac}
\definecolor{MyCiteColor}{HTML}{0099FF}
\definecolor{MyRed}{HTML}{3E186A}
\hypersetup{%
  pdfborderstyle={/S/U/W 1},
  colorlinks=true,
  linkcolor=MyBlue,
  citecolor=MyCiteColor,
  urlcolor=MyRed,
  anchorcolor=MyBlue,
  linkbordercolor=MyRed,
}
\usepackage[shortlabels]{enumitem}
\usepackage{makecell}
\usepackage[final]{pdfpages}
\setboolean{@twoside}{false}
\usepackage{pdfpages}
\usepackage{tikz}

\usetikzlibrary{shapes,arrows,spy,positioning,decorations}
\usepackage{caption}
\usepackage{subcaption}
\usepackage{mathrsfs}
\usepackage{scalefnt}
\usepackage{verbatim}
\oddsidemargin \evensidemargin
\makeatletter
\newcommand{\eqnum}{\refstepcounter{equation}\textup{\tagform@{\theequation}}}
\makeatother

\newtheorem{theorem}{Theorem}
\numberwithin{theorem}{section}
\newtheorem{proposition}[theorem]{Proposition}

\newtheorem{corollary}[theorem]{Corollary}
\newtheorem{definition}[theorem]{Definition}

\newtheorem{remark}[theorem]{Remark}
\newtheorem{example}[theorem]{Example}

\newcommand{\RR}{\mathbb{R}}
\newcommand{\QQ}{\mathbb{Q}}
\newcommand{\PP}{\mathbb{P}}
\newcommand{\CC}{\mathbb{C}}
\newcommand{\ZZ}{\mathbb{Z}}

 \date{}

\title{\textbf{Nearest Points on Toric Varieties}}

\author{Martin Helmer and Bernd Sturmfels}

\begin{document}

\maketitle

\qquad \qquad   {\em   Dedicated to Alicia Dickenstein
on the occasion  of her $60$th birthday} 
\bigskip
     
\begin{abstract} \noindent
We determine the Euclidean distance degree of a projective toric variety.
This extends the formula of Matsui and Takeuchi
for the degree of the $A$-discriminant in terms of Euler obstructions.
Our primary goal is the development of reliable
algorithmic tools  for computing the points on a 
real toric variety that are closest to a given data point. \end{abstract}

\section{Introduction}

We are interested in the best approximation of
 data points in $\RR^n$ by a model that is given by
 a monomial parametrization. Such a model corresponds to a 
projective toric variety. Our result is a formula for
the generic Euclidean distance degree (gED degree \cite{DHOST})
of that variety.

Consider the problem of identifying  $d$ unknown real numbers $t_1,t_2,\ldots,t_d$
by sampling noisy products of any $k$ of these numbers. The input data consists of 
$\binom{d}{k}$ measurements $u_{i_1 i_2 \cdots i_k}$
that are supposed to be approximations of  $\,t_{i_1} t_{i_2} \cdots t_{i_k}\,$
for $1 \leq i_1 < i_2 < \cdots < i_k \leq d$.
The least squares paradigm suggests the unconstrained polynomial optimization problem
\begin{equation}
\label{eq:hypersimplexopt1}
 \hbox{Minimize the function} \quad
L(t_1,\ldots,t_d) \,\,\,\, = 
\sum_{1 \leq i_1 < \cdots < i_k \leq d } \!\!\!\!\!
\bigl(\, t_{i_1} t_{i_2} \cdots t_{i_k}\, -\,u_{i_1 i_2 \cdots i_k} \,\bigr)^2.
\end{equation}
The critical points of this problem are solutions of the system of polynomial equations
\begin{equation}
\label{eq:hypersimplexopt2}
\frac{\partial L}{\partial t_1} \,= \,
\frac{\partial L}{\partial t_2} \,=\, \cdots \,= \,
\frac{\partial L}{\partial t_d} \,=\, 0 . 
\end{equation}
The non-zero complex solutions to (\ref{eq:hypersimplexopt2})
come in clusters of $k$ solutions that differ by multiplication with a $k$-th root of unity.
The number of such clusters for generic data $u_{i_1 i_2 \cdots i_k}$
is the algebraic degree of the optimization problem (\ref{eq:hypersimplexopt1}).
 For instance, if $d=4,k=2$ then (\ref{eq:hypersimplexopt2}) is a system
of $4$ cubics in $4$ unknowns. Using {\tt Macaulay2} \cite{M2}, we find that it has
 $28$ pairs of solutions $\{t,-t\}$.
Thus, for $d=4,k=2$, the algebraic degree of the problem (\ref{eq:hypersimplexopt1}) equals $28$.
Proposition \ref{prop:hypersimplex}   generalizes that number 
to a combinatorial formula in terms of $d$ and~$k$.

The models in this paper are as follows.
We fix an integer $d \times n$-matrix
 $A = (a_1,a_2,\ldots,a_n)$ 
of rank $d$ such that $(1,1,\ldots,1) $ lies in the row space of $A$ over $\QQ$.
We allow for $A$ to have negative entries.
Each column vector $a_i$ corresponds to a (Laurent) monomial
$t^{a_i} = t_1^{a_{1i}} t_2^{a_{2i}} \cdots t_d^{a_{di}}$.
The {\em affine toric variety} $\,\tilde X_A\,$ is the closure in
$\CC^n$ of the set $\,\{ (t^{a_1}, \ldots , t^{a_n}) \,:\,t \in (\CC^*)^d \}$,
where $\CC^* = \CC \backslash \{0\}$.
This is the affine cone over the 
{\em projective toric variety} $X_A \subset \PP^{n-1}$
with the same parametrization.
Note that ${\rm dim}(X_A) = d-1$ and ${\rm dim}(\tilde X_A) = d$.
For basics on toric geometry and toric algebra we refer to the books \cite{CLS, GBCP}.

Fix a vector $\lambda = (\lambda_1 ,\ldots,\lambda_n)$ of positive reals
and consider the $\lambda$-weighted Euclidean norm on $\RR^n$, defined by
$|| x ||_\lambda = (\sum_{i=1}^n \lambda_i x_i^2)^{1/2}$.
Given  $u \in \RR^n$, we seek to 
find a real point $v \in \tilde X_A$ that is closest to $u$.
Thus, our aim is to solve the constrained optimization problem
\begin{equation}
\label{eq:opt1}
 {\rm Minimize}\,\,\, || u-v||_\lambda \,\,\,
\hbox{subject to} \,\,\, v \in \tilde X_A \cap \RR^n.
\end{equation}
This is equivalent to the unconstrained optimization problem
\begin{equation}
\label{eq:opt2}
 {\rm Minimize}\,\,  \sum_{i=1}^n \lambda_i ( u_i - t^{a_i}  )^2 \,\,\,
\hbox{over all $\,\,t = (t_1,\ldots,t_d) \in \RR^d$. } 
\end{equation}
 The number of complex critical points of (\ref{eq:opt1})
  is denoted ${\rm EDdegree}_\lambda(X_A) $. This is the ED degree
(cf.~\cite{DHOST, ottaviani2014exact}) of the toric variety $X_A$. It 
depends on $\lambda$ but is independent of $u$, since $u$ is generic.
 It governs the intrinsic algebraic complexity of finding and
 representing the exact solutions to (\ref{eq:opt1}) and (\ref{eq:opt2}). In particular, it is an upper bound for the number of local minima.
The number of complex critical points of (\ref{eq:opt2}) is the product
 ${\rm EDdegree}_\lambda(X_A) \cdot [\ZZ^d: \ZZ A]$. The index
  arises as a factor because it is the degree
 of the monomial parametrization of~$X_A$.
 
 If the weight vector $\lambda$ is chosen generically then
  ${\rm EDdegree}_\lambda(X_A)$ is independent of $\lambda$.
  We call this the {\em generic ED degree} of the toric variety $X_A$
  and we denote it by ${\rm gEDdegree}(X_A) $.
  For instance,  in  (\ref{eq:hypersimplexopt1}) we
 saw that ${\rm gEDdegree}(X_A) = 28$ for the threefold $X_A \subset \PP^5$ given~by
 \begin{equation}
 \label{eq:octahedron} \qquad
  A \,\,= \,\,\begin{pmatrix}
 1 & 1 & 1  & 0 & 0 & 0 \\
 1 & 0 & 0  & 1 & 1  & 0 \\
 0 & 1 & 0  & 1 & 0  & 1 \\
 0 & 0 & 1  & 0 & 1  & 1 
 \end{pmatrix} \,\quad = \,\,\,\, \hbox{the octahedron.}
 \end{equation}

The following formula, inspired by Aluffi \cite{AluffiCM}, will be derived  and used in this paper.

\begin{theorem} \label{thm:eins} The generic
Euclidean distance degree of the projective toric variety $X_A$~is
 \begin{equation}
\label{eq:EDtoric}
{\rm gEDdegree}(X_A) \quad = \quad \sum_{j=0}^{d-1} (-1)^{d-j-1} \cdot (2^{j+1} - 1) \cdot V_{j},
\end{equation}
 where $V_j$ is the sum of the Chern-Mather volumes of all 
 $j$-dimensional faces of~$P = {\rm conv}(A)$.
\end{theorem}

The lattice polytope $P = {\rm conv}(A)$ has dimension $d-1$ since ${\rm rank}(A) = d$.
If the toric variety $X_A$ is smooth then $P$ is simple and $V_j$ is the sum of the
normalized lattice volumes of the $j$-faces of $P$. In the smooth case,
 Theorem \ref{thm:eins} is precisely the formula given in \cite[Corollary 5.11]{DHOST}. 
What is new here is the extension to the singular case.
Indeed, $X_A$ is an arbitrary singular projective toric variety in $\PP^{n-1}$.
In particular, $X_A$ is generally not normal.

Theorem \ref{thm:eins} rests on work by Aluffi \cite{AluffiCM}, Esterov \cite{Esterov}, and Matsui-Takeuchi \cite{MT}.
The key notion is the {\em Chern-Mather volume} (or CM volume for short).
 We will define this in Section~\ref{sec:three}. One ingredient is
 the {\em local Euler obstruction} \cite[Chapter 8]{BrasSeadeSuwa}
of singular strata on $X_A$.
We now present  a formula for the
dimension and degree of the {\em $A$-discriminant} \cite{GKZ},
that is, the variety $X_A^\vee$ projectively dual to $X_A$.
The following is a variant of \cite[Theorem 1.4]{MT}:

\begin{theorem} \label{thm:zwei}
Using notation as above, the polar degrees of the projective toric variety are
\begin{equation}
\label{eq:polartoric} \delta_i(X_A) \,\, = \,\, \sum_{j=i+1}^{d} (-1)^{d-j} \binom{j}{i+1} V_{j-1} . 
\end{equation}
The codimension of the $A$-discriminant is 
$\,{\rm min}\{c \,:\,\delta_{c-1} \not= 0\}$. For that $c$,
 ${\rm degree}(X_A^\vee) = \delta_{c-1}$.
\end{theorem}

We note that the polar degrees of projective varieties
are of independent interest in the study of algorithms for real algebraic geometry. 
They govern the complexity of methods for reliably sampling points in each connected 
component of a semi-algebraic set  (cf.~\cite{BGHP1,SafSch}).
The polar degrees $\delta_i$ can also be seen as the degrees of polar varieties.
Foundational results on this topic  can be found in the work of
Kleiman  \cite{Kleiman1}, Piene \cite{piene1978polar, piene1988cycles}
and Bank {\it et al.}~\cite{BGHP2}.

Our focus in this paper is on tools for concrete computations,
starting from an integer matrix $A$.
We implemented the formulas for the polar degrees  and the gED degree 
in {\tt Macaulay2} \cite{M2}. Given an arbitrary integer matrix $A$ as above, our software computes
the quantities in  (\ref{eq:EDtoric})-(\ref{eq:polartoric}). The code and accompanying discussion can be
found at the supplementary website
\begin{center}
\hfill \url{http://martin-helmer.com/Software/toricED.html} \hfill  \eqnum\label{eq:website}
\end{center}

For a concrete illustration consider the case $d=2$,
when $X_A$ is a toric curve in $\PP^{n-1}$.
After row operations and column permutations, we may assume
that our input has the form
$$ A \quad = \quad \begin{pmatrix}
\alpha_1 & \alpha_2 & \alpha_3 & \cdots & \alpha_{n-1} & \alpha_n \\
1 & 1 & 1 & \cdots  & 1 & 1 
\end{pmatrix}, $$
where $0\leq\alpha_1 < \alpha_2 < \alpha_3 <  \cdots < \alpha_{n-1} < \alpha_n$ and
the differences $\alpha_i-\alpha_j$  are relatively prime.
One finds that the generic ED degree of the toric curve $X_A$ equals $\, 2 \alpha_n + \alpha_{n-1} 
- \alpha_2 - 2 \alpha_1 $.
This quantity is the expected number of complex solutions to the
polynomial system
\begin{equation}
\label{eq:criticaltwo}
 \frac{1}{t} \frac{\partial{L}}{\partial{s}} \,\, = \,\, 
\frac{\partial{L}}{\partial{t}} \,\, = \,\, 0, \qquad \qquad \hbox{where}
\end{equation}
\begin{equation}
\label{eq:globalL}
 L(s,t) \quad = \quad
\lambda_1 (s^{\alpha_1} t - u_1)^2 + 
\lambda_2 (s^{\alpha_2} t - u_2)^2 + \cdots + 
\lambda_n (s^{\alpha_n} t - u_n)^2 . 
\end{equation}
A priori knowledge of the ED degree is useful for optimization because
it furnishes an upper bound on the number of local minima of $L$.
The following numerical example illustrates this.

\begin{example} \label{ex:rnc} \rm
Let $n = 7$ and $\alpha = (0,1,2,3,4,5,6)$, so $X_A$ is the
rational normal curve in $\PP^6$.  The ED degree is $16$.
 The weight vector $\lambda = (1,1,1,1,1,1,1) $ exhibits the generic behavior,
 by Proposition \ref{ref:eqholds}.   So, we
 fix unit weights and use standard Euclidean distance.

Consider the data vector $u = (11, 1, 3, 1, 3, 1, 11)$ in $\RR^7$. We seek to
find the real point on the surface $\tilde X_A \cap \RR^7$ that is
located closest to $u$. Note that
we may regard $u$ as the vector of coefficients of a binary sextic, and hence
 as a symmetric tensor of format 
$2 {\times} 2 {\times} 2 {\times} 2 {\times} 2 {\times} 2 $.  See
 \cite[\S 8]{DHOST} or \cite[\S 4]{ottaviani2014exact}.
In that interpretation, our goal would be
to find the {\em best rank~$1$ approximation of the tensor} $u$.
We do this by minimizing the squared-distance function
$$ L(s,t) \,=\,
(t-11)^2+(st-1)^2+(s^2t-3)^2+(s^3t-1)^2+(s^4t-3)^2+(s^5t-1)^2+(s^6t-11)^2 .$$
As expected, the system  (\ref{eq:criticaltwo}) has $16$ complex solutions.
Precisely  eight of these $16$ are real. By the Second Derivative Test,
four of these eight are found to be local minima. They are
$$ 
\begin{matrix}
s &  t & L(s,t) \\
1 & 4.4285714285714285714 & 125.71428571428571428 \\
4.5086875578349189693 &  
0.0012891163419679352
 & 139.66300592712833700 \\
.22179403366779357295 &  10.829114809514133306 & 139.66300592712833700 \\
-1 &  3.5714285714285714283 & 173.71428571428571429 \\
\end{matrix}
$$
The global minimum is attained at $(s,t) = \bigl(1,\, 31/7 \bigr)$,
with value $L(s,t) = 880/7 $. \hfill $\diamondsuit$
\end{example}

Section~\ref{sec:three} develops the relevant results from algebraic geometry.
After defining polar degrees, Euler obstructions, and CM volumes,
we prove Theorems~\ref{thm:eins} and \ref{thm:zwei}.
Section~\ref{sec:four} starts by illustrating
these results for toric surfaces $(d=3)$. We then focus on
toric hypersurfaces in $\PP^{n-1}$.
These are defined by a single binomial,
and their conormal varieties are toric too.
We write these in terms of a Cayley polytope, and we 
express (\ref{eq:EDtoric})-(\ref{eq:polartoric}) in terms of
the binomial's exponents. In Section~\ref{sec:five} we derive the discriminants
in $\lambda$ and $u$ whose nonvanishing ensures that
${\rm gEDdegree}(X_A)$ correctly counts the complex critical points
of (\ref{eq:opt1}). We also discuss the tropicalization of 
the conormal variety of $X_A$, along the lines of \cite{DFS, DT}. 
We end the paper by returning to its beginning:
a formula for the generic ED degree of the hypersimplex reveals
the intrinsic algebraic complexity of
learning $d$ numbers from noisy $k$-fold products.

\section{Euler Obstructions and Chern-Mather Volumes}
\label{sec:three}

The (generic) ED degree of a projective variety
$X \subset \PP^{n-1}$ is the sum of the polar degrees of $X$.
The following formula was derived in \cite[Theorem 5.4]{DHOST}
and used in  \cite[Corollary 3.2]{ottaviani2014exact}:
 \begin{equation}
\label{eq:sumofpolar}
\mathrm{gEDdegree}(X)\,\,=\,\, \, \delta_0(X)  + \delta_1(X) + \cdots + \delta_{n-1}(X).
\end{equation}
Many authors, including Fulton \cite{fulton},
Holme \cite{holme1988geometric} and Piene \cite{piene1978polar},
define $\delta_j(X)$ as the degree of the $j$-th \textit{polar variety} of $X$ with respect to a general linear subspace $\,\ell_j=\PP^{j+{\rm codim}(X)}\subset \PP^{n-1}$:
$$ P_j\,\,=\,\,\,\overline{\left\lbrace x \in X_{\mathrm{smooth}} \; | \; \dim( T_x X \cap \ell_j )\,\geq \,j+1 
\right\rbrace} \,\,\, \subset\,\,\, \PP^{n-1}.$$
Following Kleiman \cite{Kleiman1}, we can also define $\delta_j(X) $
 using the multidegree of the conormal variety ${\rm Con}(X)$.
 This approach is used in \cite{AluffiCM}. It is explained in \cite[\S 5]{DHOST} after equation (5.3).
In practice, we can use the command {\tt multidegree} in {\tt Macaulay2},
as shown in Example \ref{ex:surfaceinP5}.

If $X \subset \PP^{n-1}$ is smooth then
its polar degrees can be expressed in terms of the Chern classes
of the tangent bundle. Holme 
\cite[page 150]{holme1988geometric} and Piene \cite[Thm.~3]{piene1988cycles} 
give the formula
\begin{equation}
\label{eq:holmeformula}
 \delta_i(X) \,\,\, = \,\,\, \sum_{j=i+1}^{d} (-1)^{d-j} \cdot \binom{j}{i+1} \cdot {\rm deg}(c_{d-j}(X)).
 \end{equation}
  This formula also covers the singular case (as shown by Piene \cite{piene1988cycles}) if we replace the Chern class with the Chern-Mather class. 
 This is the approach to be pursued in this section.
 We shall develop the combinatorial meaning of the formula \eqref{eq:holmeformula} 
 in the case where $X_A$ is an arbitrary singular projective toric variety.
As a consequence, we obtain a practical algorithm,
made available in~\eqref{eq:website},
for computing the polar degrees and the generic ED degree of $X_A$.
We begin by explaining the relevant results of Esterov \cite{Esterov}
and Matsui and Takeuchi  \cite{MT}. These will enable us to
derive  Theorems~\ref{thm:eins} and \ref{thm:zwei}.
 
As above,  $A = (a_1,a_2,\ldots,a_n)$ is an integer $d \times n$-matrix of rank $d$ 
with $(1,1,\ldots,1) $ in its row space. The columns $a_i$ span the
semigroup $\mathbb{N} A$ and  the lattice $\ZZ A$, both in $\ZZ^d$.
The polytope $P = {\rm conv}(A)$ has dimension $d-1$ and it lives in $\RR^d$.
Let $\alpha$ be an $(s-1)$-dimensional face of $P$. Its span  $\RR \alpha$ is a
 linear subspace of dimension $s$ in $\RR^d$. The intersection
 $M_\alpha := \RR \alpha \cap \ZZ^d$  is a lattice of rank $s$.
  The quotient group is also free abelian:
  $ \ZZ^d/M_\alpha \,\simeq \, \ZZ^{d-s}$.
  
Let $A_\alpha$ denote the set of all  columns $a_i$ of $A$ that lie in $\alpha$.
The lattice $\ZZ A_\alpha$ spanned by that set is a subgroup of finite index in $M_\alpha$.
We also consider the image of the set of columns of $A$ in $\ZZ^d/M_\alpha$.
This is a $(d-s)$-dimensional
vector configuration,  to be denoted by $A /\alpha $.
We wish to stress that
the toric varieties in this paper are generally not normal,
and all our volumes are understood in the normalized integer sense 
that is customary in toric geometry. 

\begin{definition} \label{def:subdiagram} \rm
Fix two faces $\alpha, \beta$ of $P$ such that $\beta \subset \alpha$. After a change of 
coordinates, we may assume that the origin in $\ZZ^d$ is contained in the face $\beta$. 
We write $A_\alpha/\beta$ for the image of the finite set $A_\alpha$
in the  free abelian group $M_\alpha/M_\beta$. Its convex hull
${\rm conv}(A_\alpha/\beta)$ is a polytope of  dimension
$r = \dim(\alpha)-\dim(\beta)$ in the real vector space
$(M_\alpha/M_\beta) \otimes_\ZZ \RR = \RR \alpha/\RR \beta \simeq \RR^{r}$.

We define the
 {\em subdiagram volume} of $\beta$ in $\alpha$ to be the positive integer
 \begin{equation}
 \label{eq:SDV}
\mu(\alpha/\beta)\,\,=\,\,{\rm Vol} \bigl(\,
{\rm conv}(A_\alpha/\beta) \setminus
{\rm conv}((A_\alpha/\beta )\backslash \{0\}) \,\bigr) 
\end{equation}
where ${\rm Vol}$ is the $r$-dimensional volume that is normalized with respect to the 
 lattice $M_{\alpha}/M_\beta$.
\end{definition}

The notion of subdiagram volume is also defined in \cite[Definition 3.8]{GKZ} and in \cite[Definition 4.5]{MT},
but their notation and normalization conventions are  slightly different.

\begin{remark} \label{rmk:adapted} \rm
To compute the subdiagram volume in (\ref{eq:SDV}), we 
use coordinates on $ \ZZ^d$ that are adapted to the inclusions
$ M_\beta \subset  M_{\alpha} \subset \ZZ^d$.
Changing coordinates on $\ZZ^d$ corresponds to 
integer row operations on $A$. We
shall use the following procedure to carry this out:
\begin{itemize}
\item First reorder the columns of $A$ so that
those in $\beta$ come first, followed by those in $\alpha \backslash \beta$, and the remaining columns last. In other words, we write $A$ in block form as 
$$A \,\,= \,\,\bigl( A_{\beta}, \,A_{\alpha \backslash \beta}, \,A_{P \backslash \alpha}\bigr).$$
\item Next compute the {\em Hermite normal form}  
of $A$.  It has the triangular block structure
$$ A' \quad = \quad 
\bordermatrix{ & \,\,\beta &\, \alpha \backslash \beta & P \backslash \alpha \cr
 &\,\,  * & * & *  \cr
 &\, \,0 & C & *  \cr
&\, \,0 & 0 & * \cr }.
$$
\end{itemize} 
Note that $X_A = X_{A'}$. The integer matrix $C$ has $r$ rows 
where $r = \dim(\alpha)-\dim(\beta)$.
Restricting to these $r$ rows corresponds to the appropriate projection
 $\ZZ^n \rightarrow    \ZZ^r \simeq M_\alpha/M_\beta$.                                                                                                                                                                                                                                                                                                                                                                                                                                                                                      
  To find the subdiagram volume in (\ref{eq:SDV}), we 
  may use the normalized $r$-dimensional volumes of 
  the polytopes ${\rm conv}(C \cup \{0\}) $ and ${\rm conv}(C)$. 
  These considerations imply  the following formula:
\begin{equation}
\label{eq:SDV2}
 \mu(\alpha/\beta) \,\,= \,\,{\rm Vol}({\rm conv}(C \cup \{0\}) ) - {\rm Vol}({\rm conv}(C)). 
 \end{equation}
\end{remark}

MacPherson \cite{macpherson1974chern} introduced the local Euler obstructions in singularity theory.
See the book \cite{BrasSeadeSuwa} for subsequent developments.
Ernstr{\"o}m \cite{ernstrom1997plucker} related this to 
polar degrees and dual varieties.
For the case of toric varieties, the local Euler obstructions admit a combinatorial 
description in terms  of
subdiagram volumes. This was developed by Esterov \cite[\S2.5]{Esterov} and refined by Matsui and Takeuchi 
\cite[\S4.2]{MT}.
We shall present a review of these results, modified to use the notation above.
The matrix $A$ and the polytope $P = {\rm conv}(A)$ are as before.

 \begin{definition} \label{def:Eu}   \rm
Let $\beta$ be a face of $P$. The {\em Euler obstruction} of $\beta$ is an integer
${\rm Eu}(\beta)$ that depends on the point configuration $A$.
 It is defined recursively by the following relations:
\begin{enumerate}
\item ${\rm Eu}(P)\,\,=\,\,1,$
\item $\displaystyle {\rm Eu}({\beta}) \quad =\sum_{{\alpha} {\rm \;s.t.}\; {\beta}{\;\rm is\; a} \atop
{\rm proper\; face \; of \;} {\alpha}} 
\!\!\!\!
(-1)^{\dim({\alpha})-\dim({\beta})-1}\cdot \mu(\alpha/\beta) \cdot {\rm Eu}(\alpha).$
\end{enumerate}
If $X_A$ is smooth along the orbit given by the face $\beta$ then
${\rm Eu}({\beta})=1$. We note that, as discussed above, the lattice indices in
\cite[Theorem 4.7]{MT} are subsumed in Definition \ref{def:subdiagram}.
See also \cite[Corollary 1.11.3]{Nod}.
\end{definition}

Let $\beta$ be a face of $P= {\rm conv}(A)$ and $T_{\beta}$ the
corresponding orbit. Let ${\rm Eu}_{X_{A}}:X_A\to \ZZ$ be the local Euler obstruction of $X_A$ as defined by \cite{macpherson1974chern} and \cite[Chapter 8]{BrasSeadeSuwa}.
Note that ${\rm Eu}_{X_{A}}$ is constant on the orbits given by the faces of $P$. Let ${\rm Eu}_{X_{A}}(T_{\beta})$ denote the value of ${\rm Eu}_{X_{A}} $ for any point in $T_{\beta}$. By Theorem 4.7 of Matsui and Takeuchi \cite{MT} we have that 
\begin{equation}
{\rm Eu}_{X_{A}}(T_{\beta})= {\rm Eu}(\beta)\cdot [M_{\beta}:\ZZ A_{\beta}] .\label{eq:EulerObsturctionValueOnOribit}
\end{equation}
Using the Euler obstruction of Definition \ref{def:Eu}, we now define the Chern-Mather (CM) volume.

\begin{definition} \label{def:Vj} \rm The {\em Chern-Mather volume} of a face $\beta$ of $P$
is an integer that depends on~$A$. It is the product
${\rm Vol}(\beta){\rm Eu}(\beta)$  of the normalized volume and the Euler obstruction of 
$\beta$.
As in Theorem~\ref{thm:eins},
we write $V_j$ for the sum of the CM volumes of the
 $j$-dimensional faces of~$P$:
 \begin{equation}
 \label{eq:Vjay}
V_j \quad =\sum_{\beta {\rm \; face \; of\;} P \atop  \dim(\beta)=j}
\! {\rm Vol}(\beta){\rm Eu}(\beta).
\end{equation}
We chose to use the term ``volume'' even though
the integers ${\rm Eu}(\beta)$ and $V_j$ can be negative.
\end{definition}

\begin{remark} \rm
The primary aim of Matsui and Takeuchi in \cite{MT} was to compute
the dimension and degree of the $A$-discriminant  $X_A^\vee$.
These are given by the first non-zero polar degree:
if $\delta_0 = \cdots = \delta_{c-2} = 0$
and $\delta_{c-1} > 0$ then ${\rm codim}(X_A^\vee) = c$
and ${\rm degree}(X_A^\vee) = \delta_{c-1}$.
This is essentially the content of \cite[Theorem 1.4]{MT}.
However, it is important to note that the quantities
$\delta_\bullet$ in \cite[(1.6)]{MT}
are \underbar{not} the polar degrees of $X_A$. Instead, they
are the alternating~sums 
$$
\delta_0 \,,\,\,
\delta_1 - 2 \delta_0 \,, \,\,
\delta_2 - 2 \delta_1 + 3 \delta_0\,,\,\,
\delta_3 - 2 \delta_2 + 3 \delta_1 - 4 \delta_0 \,, \,\,
\delta_4 - 2 \delta_3 + 3 \delta_2 - 4 \delta_1 + 5 \delta_0\,,\, \,\ldots.
$$
Note that the first non-zero number in this list also gives
the codimension and degree of  $X_A^\vee$.

We prefer the direct formulation, just using the polar degrees, given
in the second and third sentence of Theorem~\ref{thm:zwei}.
Formula (\ref{eq:polartoric}) writes the polar degrees in terms of
CM volumes.
\end{remark}

\begin{proof}[Proof of Theorem \ref{thm:zwei}]
For any subvariety $X$ of $\PP^{n-1}$,
the $i^{th}$ polar degree can be expressed in terms of 
the Euler obstructions of linear sections of $X$.
Ernstr\"om  \cite[Theorem 2.2]{ernstrom1997plucker} proves
\begin{equation}
\delta_{i}(X)\,\,=\,\,(-1)^{{\rm dim}(X)-i} \left( \chi( {\rm Eu}_{X^{(i)}})
-2\chi( {\rm Eu}_{X^{(i+1)}})+\chi( {\rm Eu}_{X^{(i+2)}}) \right), \label{eq:ErnstromPolarDegree}
\end{equation} where $X^{(j)}=X \cap H_1 \cap \cdots \cap H_j$ for general hyperplanes $H_{\ell}$ in $\PP^{n-1}$.
In their proof of \cite[Theorem 1.4]{MT},  Matsui and Takeuchi give an
explicit expression for the terms in \eqref{eq:ErnstromPolarDegree}
when $X=X_A$ and ${\rm dim}(X)=d-1$.
 Specifically, the equations (3.16) and (3.10) in \cite{MT} show that
 \begin{equation}
\chi( {\rm Eu}_{X_A^{(0)}})\,\,\,=\,\,\,\chi( {\rm Eu}_{X_A})\,\,=\,\,V_0 \qquad {\rm and} \qquad \qquad
\label{eq:MatsuiTakeuchiEulerObstructionZero}
\end{equation}
\begin{equation} \qquad 
\chi( {\rm Eu}_{X_A^{(i)}})\,\,=\,\,\sum_{j=i}^{d-1} (-1)^{j-i}{j-1 \choose i-1} V_j \quad  {\rm for\;} 
\,i=1, \ldots, d-1. \label{eq:MatsuiTakeuchiEulerObstruction}
\end{equation} 
Substituting \eqref{eq:MatsuiTakeuchiEulerObstructionZero} and
 \eqref{eq:MatsuiTakeuchiEulerObstruction}
 into \eqref{eq:ErnstromPolarDegree}  gives the formula
 $$
\delta_{0}(X_A)\,\,=\,\,
(-1)^{d-1} \left(V_0 -2\sum_{j=1}^{d-1} (-1)^{j-1} V_j + \sum_{j=2}^{d-1} (-1)^{j}{(j-1)} V_j \right).
$$
Similarly, for $i=1,\ldots, d-1$ we obtain
$$
\begin{matrix}
\delta_{i}(X_A)\,=\,(-1)^{d-1} \left( \sum_{j=i}^{d-1} (-1)^{j}{j-1 \choose i-1} 
V_j-2\sum_{j=i+1}^{d-1} (-1)^{j-1}{j-1 \choose i} V_j 
+ \sum_{j=i+2}^{d-1} (-1)^{j}{j-1 \choose i+1}  V_j \right). 
\end{matrix}
$$
By reindexing the two summations above, and by collecting terms, we
obtain the more compact expression for the polar degrees 
given in (\ref{eq:polartoric}). This completes the proof.
\end{proof}

\begin{proof}[Proof of Theorem \ref{thm:eins}]
This follows from Theorem \ref{thm:zwei} using the formula (\ref{eq:sumofpolar}).
\end{proof}

We next justify why we chose the term ``Chern-Mather volume'' for the quantities $V_j$ 
in Definition \ref{def:Vj}.  The {\em Chern-Mather class}  is a generalization 
of the total Chern class (of the tangent bundle) to singular varieties. 
See \cite[Section 10.6]{BrasSeadeSuwa}
 or \cite[Example 4.29]{fulton} 
for the definition.
Piene \cite{piene1988cycles} expressed
 the Chern-Mather class of a projective variety as an alternating sum of polar degrees. 
  Her formula leads to the following identification of the Chern-Mather class of a
  toric variety $X_A$  with the Chern-Mather volumes~$V_j$ of its matrix $A$.
  We regard the Chern-Mather class of $X_A$ as an element in 
the Chow ring $A^*(\PP^{n-1}) \cong \ZZ[h]/ \langle h^n \rangle$ 
of the ambient projective space $\PP^{n-1}$.
Here $h$ denotes the hyperplane class.

\begin{proposition} \label{thm:CMClass}
The Chern-Mather class of the projective toric variety $X_A \subset \PP^{n-1}$ equals
\begin{equation}\label{eq:CMToric}
c_M(X_A)\,\,=\,\,\sum_{j=0}^{d-1} V_j \cdot h^{n-j-1}\,\,\, \in \,\,A^*(\PP^{n-1})\,\cong 
\,\ZZ[h]/ \langle h^n \rangle.
\end{equation}
 In particular, the CM volume $V_j$ is the degree of the dimension $j$
  Chern-Mather class of $X_A$.
\end{proposition}

\begin{proof} In light of Theorem \ref{thm:zwei}, this follows immediately from
Piene's formula  \cite[Theorem 3]{piene1988cycles} for the Chern-Mather class 
of a projective variety in terms of polar degrees. The
  simplification of the summations required to 
  arrive at the formula (\ref{eq:CMToric}) is aided  considerably by
    employing the Chern-Mather involution formulas of Aluffi \cite{AluffiCM}.
\end{proof}

The result of Proposition \ref{thm:CMClass} may also be expressed in the Chow ring of $X_A$ as \begin{equation}\label{eq:CM_In_Toric_CR}
c_M(X_A)\,\,\,\,=\,\sum_{\alpha \; {\rm a \;face \; of} \,P} {\rm Eu}(\alpha)\cdot [M_{\alpha}:\ZZ A_{\alpha}] [\overline{T_{\alpha}}]\,\,\, \in \,\,A^*(X_A),
\end{equation} where $[\overline{T_{\alpha}}]$ is the class in $A^*(X_A)$ of the orbit closure associated to a face $\alpha$ of $P$. This reformulation follows from Proposition \ref{thm:CMClass} and \eqref{eq:EulerObsturctionValueOnOribit}.
A direct proof  is given in \cite[Theorem~2]{Piene2016CM}.

Theorem~\ref{thm:eins} is now a special case of \cite[Proposition 2.9]{AluffiCM}. 
Aluffi's result  expresses the ED degree of an arbitrary projective variety
 in terms of the Chern-Mather class.
 While this does encompass our situation, 
it does not provide new tools for actually  computing polar degrees, 
Chern-Mather classes, or ED degrees.
 Our contribution fills this gap in the toric case.
  We furnish an algorithm for computing these quantities for
 an arbitrary projective toric variety $X_A$, not necessarily normal.
Our method is implemented in the {\tt Macaulay2} package  at (\ref{eq:website}).
Its input is the $d \times n$-integer matrix $A$, and its output is the 
numbers in (\ref{eq:EDtoric}) and~(\ref{eq:polartoric}).

Our  implementation allows for relatively efficient and extremely scalable computation.
The running time is almost entirely determined by the facial structure of
$P = {\rm conv}(A)$. While this may make the computation difficult for
  high-dimensional polytopes with many faces, it has several important advantages over algebraic methods. First, the running time of our code
  has very little direct dependence on the degree of $X_A$.
  For algebraic methods   (both numerical and symbolic), this will be a bottleneck:
   computations  become infeasible as ${\rm degree}(X_A)$ grows. 
   Second, for fixed $d$ and large $n$, the toric ideal of $A$
   can become unmanageable quite rapidly, while an iteration over the faces
   of $P$ is still feasible. Third,
      our combinatorial method is exact, and many portions of the computation 
      could be parallelized.

We close this section by summarizing the steps of our algorithm.
The input is the matrix $A$. It computes
 the CM volume for each face of $P = {\rm conv}(A)$.
 The output is the list of CM volumes
 $V_0,\ldots,V_{d-1}$,  the polar degrees
  $\delta_0(X_A), \ldots, \delta_{d-1}(X_A)$, and the ED degree of $X_A$.
      \begin{itemize}
\item Compute the face poset $\mathcal{P}$ of the lattice polytope $P = {\rm conv}(A)$.
\item Build a second poset $ \overline{\mathcal{P}}$, isomorphic to $\mathcal{P}$,
whose elements are the pairs $(\alpha, A_{\alpha})$ for~$\alpha\in \mathcal{P}$
\item For each chain $(P, A)\supset (\alpha_1, A_{\alpha_1}) \supset \cdots \supset (\alpha_{\ell}, A_{\alpha_{\ell}}) $ in the poset $ \overline{\mathcal{P}}$, do the following:
\begin{itemize}
\item Reorder the columns of the matrix $A$ according to this chain. The new matrix is
 $$
\tilde{A}\,\,=
\,\, \bigl(\,A_{\alpha_{\ell}}, \,A_{\alpha_{\ell-1} \backslash \alpha_{\ell}},
\, A_{\alpha_{\ell-2} \backslash \alpha_{\ell-1}}, \,
\ldots, \,A_{\alpha_{1}\backslash \alpha_{2}}, \,A_{ P \backslash \alpha_{1}}
\,\bigr).
$$
\item Find the Hermite normal form $A'$ of $\tilde{A}$, as in Remark \ref{rmk:adapted}.
\item For all pairs $1 \leq i < j \leq \ell$,
 compute the relative subdiagram volumes $\mu(\alpha_i \backslash \alpha_j)$,
  using \eqref{eq:SDV2} by selecting the appropriate submatrix $C$ of $A'$.
\end{itemize} 
\item Compute the normalized volumes of all elements in the face poset $\mathcal{P}$. 
\item Combining all subdiagram volumes and face volumes found above,
we now compute the Euler obstruction for each face of $P$ using the formula in Definition \ref{def:Eu}.
\item Compute $V_j$ using formula (\ref{eq:Vjay}).
Compute $\delta_i(X_A)$ using 
(\ref{eq:polartoric}). Output ${\rm gEDdegree}(X_A)$.
\end{itemize}

\section{Dimension Two and Codimension One}
\label{sec:four}

In this section we compute the gED degree for
instances of low dimension  and low codimension.
We start with    toric surfaces.
 Here $d{=}3$ and we assume that the matrix has the~form
 $$ A \quad = \quad \begin{pmatrix}
\alpha_1 & \alpha_2 & \alpha_3 & \cdots & \alpha_{n-1} & \alpha_n \\
\beta_1 & \beta_2 & \beta_3 & \cdots & \beta_{n-1} & \beta_n \\
1 & 1 & 1 & \cdots  & 1 & 1 
\end{pmatrix}. $$
The lattice polygon  $P = {\rm conv}(A)$
has normalized area $V_2 = {\rm Vol}(P)$.
Its polar degrees  are
\begin{equation}
\label{eq:polardim2}
\delta_0 \, =\, 3 V_2  - 2 V_1 + V_0  \,, \,\,\,
\delta_1\, = \,3 V_2 - V_1 \,\,\, \hbox{and} \,\,\,
\delta_2\, =\,  V_2. 
\end{equation}
The generic ED degree
is equal to the sum of the polar degrees:
\begin{equation}
\label{eq:EDdim2} {\rm gEDdegree}(X_A) \,\,\, = \,\,\,
\delta_0 + \delta_1  + \delta_2 \,\,\,=\,\,\,
\,\,\, 7 V_2 \,- \,3 V_1 \,+\,V_0. 
\end{equation}

If $X_A$ is smooth then $V_0$ and $V_1$ are positive integers. Namely,
$V_0$ is the number of vertices of $P$, and
  $V_1$ is number of all lattice points in the boundary of $P$.
 Here is a simple example.

\begin{example} \label{ex:P1P1}  \rm
Let $n=9$ and $X_A = \PP^1 \times \PP^1$, embedded in
$\PP^8$ with the line bundle $\mathcal{O}(2,2)$:
 $$ A \quad = \quad \begin{pmatrix}
 0 & 0 & 0 &  1 & 1 & 1 & 2 & 2 & 2 \\
  0 & 1 & 2 & 0 & 1 & 2 & 0 & 1 & 2 \\
 1 & 1 & 1 &  1 & 1 & 1 &  1 & 1 & 1 
 \end{pmatrix}
 $$
This corresponds to   approximating a data vector $u \in \RR^9$ by
  biquadratic monomials.
  Then $P  = {\rm conv}(A)$ is a  square of side length $2$.
The face volumes are $V_2 = 8$, $V_1 = 8$ and $V_0 = 4$,
and hence ${\rm gEDdegree}(X_A) = 36$. 
For instance, if the weights are
$\lambda = (4, 1, 9, 2, 3, 1, 7, 6, 5)$ and
data point is $u=(29, 14, 46, 13, -5, 42, 42, 5, 23)$
then precisely $14$ of the $36$ complex critical points are real.
This choice of $\lambda $ exhibits the generic behavior.
The ED degree drops from $36$ to $20$ if we take
$\lambda = (1,1,1,1,1,1,1,1,1)$; here the unit weights are not generic.
This degree drop is explained by the criterion we shall derive in 
 Proposition \ref{ref:eqholds}.
\hfill $\diamondsuit$
\end{example}

For singular toric surfaces $X_A$, we must consider the
CM volumes of the edges and vertices of the planar 
configuration $A$. If $X_A$ is normal  then the following
formula can be used:

\begin{corollary} \label{prop:surfaces}
Suppose that $X_A$ is a toric surface with isolated singularities in $\PP^n$.
Then $V_1$ is the number of lattice points in the boundary of
$P = {\rm conv}(A)$, and the CM volume
of a vertex $a_i$ of $A$ equals $\,
{\rm Vol}({\rm conv}(A \backslash \{a_i\}))+2- {\rm Vol}(P)$,
where ${\rm Vol}$ denotes normalized area.
Hence $V_0$ is the sum of these (possibly negative) integers, as
$a_i$ ranges over all vertices of $P$.
\end{corollary}

\begin{proof}
This follows from the general results in Section~\ref{sec:three}.
See also \cite[Proposition 1.11.7]{Nod}.
\end{proof}

The following example  illustrates Corollary~\ref{prop:surfaces}.
For a non-normal case see Example~\ref{ex:quadrangle}.
For any such small instance $A$,
we can always verify our combinatorial computation of toric ED degrees
using the general algebraic method  in \cite[(5.3)]{DHOST}.
This is done by first computing the bigraded prime ideal of the
{\em conormal variety} ${\rm Con}(X_A)$.
Recall that ${\rm Con}(X_A)$ is an irreducible closed subvariety of
dimension $n-2$ in $\PP^{n-1} \times \PP^{n-1}$. It is the closure of
the set of pairs $(x,y)$ in  $\PP^{n-1} \times \PP^{n-1}$ 
such that $x$ is a smooth point in $X_A$ and $y$ is a hyperplane tangent to $X_A$ at $x$.
The projection of ${\rm Con}(X_A)$ onto the second factor
is the {\em $A$-discriminant} $X_A^\vee$.

\begin{example} \label{ex:surfaceinP5}
\rm
Let $n=6$ and let $X_A$ be the normal toric surface in $\PP^5$ given by
 $$ A \quad = \quad \begin{pmatrix}
 1 & 0 & 1 & 2 & 3 & 1  \\
 0 & 1 & 1 & 1 & 1 & 2  \\
 1 & 1 & 1 & 1 & 1 & 1
 \end{pmatrix}.
 $$ 
 This is the closure of the image of
    $(\CC^*)^3\to \PP^5 , \,(s,t,u) \mapsto (su:tu: stu: s^2tu:s^3tu:st^2u  )$.
   
\begin{figure}[h!]
 \centering
    \begin{tikzpicture}[scale=1.7]
\draw [purple,dashed](0,1) -- (3,1);
\draw [purple,dashed](1,2) -- (1,0);
\draw [purple,dashed](1,2) -- (2,1);
\draw [purple,dashed](1,0) -- (2,1);
\draw [blue,very thick](0,1) -- (1,0);
\draw [blue,very thick](0,1) -- (1,2);
\draw [blue,very thick](1,0) -- (3,1);
\draw [blue,very thick](1,2) -- (3,1);

\node at (1,-.2) {$a_1$};
\node at (-0.15,1.15) {$a_2$};
\node at (1.15,1.15) {$a_3$};
\node at (2.15,1.15) {$a_4$};
\node at (3.15,1.15) {$a_5$};
\node at (1,2.2) {$a_6$};
\fill[teal] (0,1) circle[radius=2pt];
\fill[teal] (3,1) circle[radius=2pt];
\fill[teal] (1,2) circle[radius=2pt];
\fill[teal] (1,0) circle[radius=2pt];
\fill[teal] (1,1) circle[radius=2pt];
\fill[teal] (2,1) circle[radius=2pt];
\end{tikzpicture}
     \caption{The polygon $P  = {\rm conv}(A)$ has normalized
     area six. The only lattice points in its boundary are
     the four vertices.      Their CM volumes can be read off
     from this triangulation. \label{fig:quadrangle}
     }
\end{figure}
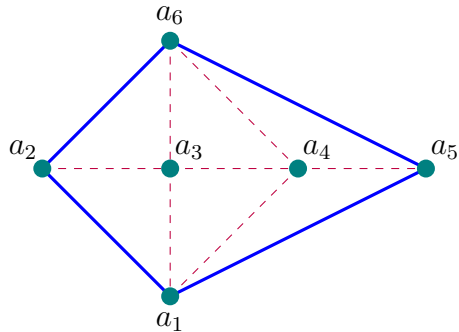
 Figure \ref{fig:quadrangle}  shows that $V_2 = 6$ and $V_1 = 4$.
The four vertices of the polygon $P$ are $a_1,a_2,a_5,a_6$, and the corresponding
complementary areas ${ \rm Vol}({\rm conv}(A \backslash \{a_i\}))$ are $3,4,4,3$.
Hence the CM volumes of the vertices are $-1,0,0,-1$,
for total of $V_0 = -2$. We conclude 
$$ {\rm gEDdegree}(X_A) \,\, = \,\, 7 V_2 - 3 V_1 + V_0 \,\,=\,\,
7\cdot 6 - 3 \cdot 4 + (-2) \,=\, 28. $$

We verify this by computing the conormal variety ${\rm Con}(X_A) \subset \PP^5 \times \PP^5$.
Each point $y \in X_A^\vee$ represents a singular curve
$\{y_1 su + y_2 tu + y_3 stu +  y_4 s^2tu + y_5 s^3tu + y_6 st^2u = 0\}$
on the toric surface $X_A \subset \PP^5$, and $x = (su:tu:\cdots:st^2u)$ is the singular point.
The {conormal variety} has dimension~$4$. Its
prime ideal $\mathcal{C}$ is minimally generated by $17$ polynomials in
the $6+6$ homogeneous coordinates of $\PP^5 \times \PP^5$. Among these are 
four binomial quadrics that generate the toric ideal of $X_A$.
The polar degrees are the coefficients of the {\em multidegree} of the ideal $\mathcal{C}$,
and they are  $\delta_0 = 8$, $\delta_1 = 14$, and $\delta_2 = 6$.
This is consistent with Theorem  \ref{thm:zwei}, which says that
$\delta_0 = 3 V_2  - 2 V_1 + V_0 $, 
$\,\delta_1=3 V_2  - V_1 $ and
$\delta_2=V_2$.
The $A$-discriminant $X_A^\vee$ is a hypersurface
of degree $8$. Its defining polynomial
is found among our $17$ ideal generators.

The following code in {\tt Macaulay2} \cite{M2} realizes
what is described in the previous paragraph.
\begin{verbatim}
R = QQ[s,t,u,x1,x2,x3,x4,x5,x6,y1,y2,y3,y4,y5,y6,Degrees=>{{1,1},{1,1},{1,1},
  {1,0},{1,0},{1,0},{1,0},{1,0},{1,0}, {0,1},{0,1},{0,1},{0,1},{0,1},{0,1}}];
f =  y1*s*u + y2*t*u + y3*s*t*u + y4*s^2*t*u + y5*s^3*t*u + y6*s*t^2*u;
I = ideal(diff(s,f),diff(t,f),diff(u,f),
     x1-s*u,  x2-t*u,  x3-s*t*u,  x4-s^2*t*u,  x5-s^3*t*u,  x6-s*t^2*u);
C = eliminate({s,t,u},I);
C = saturate(C,ideal(x1*x2*x3*x4*x5*x6));
C = saturate(C,ideal(y1*y2*y3*y4*y5*y6));
apply(first entries mingens(C),t->degree(t))
multidegree C
\end{verbatim}
The output of the last line is the binary form whose coefficients are the polar degrees.
\hfill $\diamondsuit$
\end{example}

We next examine toric hypersurfaces.
Let $X_A \subset \PP^{n-1}$ be  defined by one binomial equation
\begin{equation}
\label{eq:onebinomial}
  x_1^{c_1} \cdots x_r^{c_r} \,=\,
x_{r+1}^{c_{r+1}} \cdots x_n^{c_n} .
\end{equation}
Here $c_1,\ldots,c_n$ are positive integers that are relatively prime, and  they satisfy
\begin{equation}
\label{eq:degofhyper}
  c_1 + \cdots + c_r \, = \, c_{r+1} + \cdots + c_n \,\, = \,\, {\rm deg}(X_A) . 
  \end{equation}
Our goal is to express the gED degree and the polar degrees 
of $X_A$  in terms of $c_1,c_2,\ldots,c_n$.

The integer matrix $A$ has format $(n-1) \times n$,
and its kernel is spanned by the column vector
$(c_1,\ldots,c_r,-c_{r+1},\ldots,-c_n)^T$.
The associated lattice polytope $P = {\rm conv}(A)$ has dimension $n-2$,
and it has $n$ vertices provided $2 \leq r \leq n-2$.
We consider the {\em Cayley polytope} of $P$ and its mirror image $-P$.
This is the $(n-1)$-dimensional polytope
obtained by placing $P$ and $-P$ into parallel hyperplanes
and taking the convex hull.  See e.g.~\cite[Definition 4.6.1]{MS}.
The integer matrix representing the Cayley polytope has format
$n \times 2n$. It equals
$$ {\rm Cay}(A,-A) \quad = \quad
\begin{pmatrix} {\bf 1} &\phantom{-} {\bf 0} \\
A & -A \\ \end{pmatrix}, $$
where ${\bf 1} = (1,1,\ldots,1)$ and ${\bf 0} = (0,0,\ldots,0)$ in $\RR^{n}$.
We shall first derive the following result.

\begin{theorem}
\label{thm:codim1}
The conormal variety ${\rm Con}(X_A)$ is a
 toric variety of dimension $n-2$ in  
 $\PP^{n-1} \times \PP^{n-1}$. It corresponds to the toric variety of ${\rm Cay}(A,-A)$.
The generic ED degree of~$X_A$ is the normalized volume of
the Cayley polytope. The polar degrees $\delta_i = \delta_i(X_A)$ are given~by
\begin{equation}
\label{eq:volumepoly}
{\rm Vol} \bigl(\lambda P + \mu (-P) \bigr) \,\, = \,\,
\sum_{i=0}^{n-2} \delta_i \binom{n-2}{i} \lambda^i \mu^{n-2-i}, \;\;\; {\rm where}\; \lambda, \mu \in \RR_{>0}.
\end{equation}
\end{theorem}

The volume in (\ref{eq:volumepoly}) is
 the normalized lattice volume. Hence
$\delta_0 = \delta_{n-2} = {\rm Vol}(P)$ is the integer in
(\ref{eq:degofhyper}).
The formula (\ref{eq:volumepoly}) confirms the known fact that the 
polar degrees of a toric hypersurface are symmetric, i.e.~$\delta_{i-1} = \delta_{n-1-i}$ for all $i$.
 This symmetry of the polar degrees holds for any self-dual projective variety.
 This is known by results of
     Kleiman \cite{Kleiman1}; see also \cite{AluffiCM}.
Before we give the proof of Theorem \ref{thm:codim1}, let us present one
corollary and one example.

\begin{corollary} \label{cor:PL}
The polar degrees of $X_A$
are piecewise linear functions of $c_1,\ldots,c_n$.
Their regions of linearity are the cones in
the arrangement of hyperplanes
given by equating a subsum of $\{c_1,\ldots,c_r\}$
with a subsum of $\{c_{r+1},\ldots,c_n\}$,
 inside the $(n{-}1)$-space given by~(\ref{eq:degofhyper}).
\end{corollary}

\begin{proof}
The kernel of the matrix ${\rm Cay}(A,-A)$ is the row span of the $n \times 2n$-matrix
\setcounter{MaxMatrixCols}{20}
\begin{equation}
\label{eq:gale}
\begin{small}
\begin{pmatrix}
c_1 & c_2 & c_3 & \cdots & c_r & -c_{r-1} & \cdots & -c_n & &
0 & 0 & 0 & \cdots & 0 & 0 & \cdots & 0 \\
1 & -1 & 0 &  \cdots & 0 & 0 & \cdots & 0 & &
1 & -1 & 0 & \cdots & 0 & 0 & \cdots & 0 \\
0 & 1 & -1 &  \cdots & 0 & 0 & \cdots & 0 & &
0 & 1 & -1 &  \cdots & 0 & 0 & \cdots & 0 \\
  &  & \ddots & \ddots &  &  &  &  & &
 &  & \ddots & \ddots &      \\
  0 & 0 &  \cdots & 1 & -1 & 0 & \cdots & 0 & 
 & 0 & 0 &  \cdots & 1 & -1 & 0 & \cdots & 0  \\
 & & & &  \ddots & \ddots   &  &  & &
 & & & &  \ddots & \ddots   &      \\
&  & & & &  \ddots & \ddots   &  &  & &
 & & & &  \ddots & \ddots   &    \\
 0 & 0 & 0 & & & & 1 & \! -1 & & 
 0 & 0 & 0 & & & & 1 & \! -1 
\end{pmatrix}.
\end{small}
\end{equation}
Each of the $\binom{2n}{n}$ maximal minors of this {\em Gale dual} matrix is the difference of a
subsum of $\{c_1,\ldots,c_r\}$ and a subsum of $\{c_{r+1},\ldots,c_n\}$.
All $2^n-1$ non-zero such linear forms arise. They define
 hyperplanes inside the $(n{-}1)$-space defined by~(\ref{eq:degofhyper}).
 We restrict this hyperplane arrangement to $\RR^{n}_{> 0}$.
 Up to sign, the maximal minors of the matrix (\ref{eq:gale}) are also the maximal minors
of ${\rm Cay}(A,-A)$.  Hence the oriented matroid of ${\rm Cay}(A,-A)$
is fixed when $(c_1,\ldots,c_n)$ ranges over any cone of our arrangement
in $\RR^{n}_{>0}$.
The volume of the Cayley polytope is a sum of certain maximal minors,
selected by the oriented matroid. This implies our claim.
\end{proof}

\begin{example} \rm
\label{ex:quadrangle}
Let $n=4$ and consider the toric surface
$X_A = \{x_1^{c_1} x_2^{c_2} =x_3^{c_3} x_4^{c_4}\}$
in $\PP^3$. Writing $y_1,y_2,y_3,y_4$ for the coordinates
of the dual $\PP^3$, the conormal variety ${\rm Con}(X_A)$ is the irreducible surface
in $\PP^3 \times \PP^3$  that is defined by
$x_1^{c_1} x_2^{c_2} =x_3^{c_3} x_4^{c_4}$ together with the constraint
\begin{equation}
\label{eq:rank1a}
{\rm rank}
\begin{pmatrix}
c_1 x_1^{c_1-1} x_2^{c_2} & 
c_2 x_1^{c_1}    x_2^{c_2-1} & 
c_3 x_3^{c_3-1} x_4^{c_4} & 
c_4 x_3^{c_3}    x_4^{c_4-1} \\ 
y_1 & y_2 & y_3 & y_4 
\end{pmatrix} 
\,\,\leq \, 1.
\end{equation}
This binomial ideal is not prime, but we must saturate 
with respect to $  x_1x_2x_3x_4$ in order to 
compute the prime ideal of ${\rm Con}(X_A)$.
Performing this saturation one obtains the $2 \times 2$-minors
of the following matrix which has the same row space as the matrix above:
\begin{equation}
\label{eq:rank1b}
{\rm rank} 
\begin{pmatrix}
c_1 &  c_2 &  c_3 & c_4 \\
x_1 y_1 & x_2 y_2 & x_3 y_3 &  x_4 y_4 
\end{pmatrix} 
\,\,\leq \, 1.
\end{equation}
After replacing each variable $y_i$ by $c_i y_i$, 
we obtain the binomials corresponding to the rows~of
the $4 \times 8$-matrix in (\ref{eq:gale}).
For instance, the second row of this matrix corresponds to the
binomial $c_1 x_2 y_2 - c_2 x_1 y_1$.
The Gale dual ${\rm Cay}(A,-A)$ of (\ref{eq:gale}) represents
the $3$-dimensional polytope obtained by
taking the quadrangle $P = {\rm conv}(A)$
and placing its mirror image $-P$ on a parallel plane in $3$-space.
The volume of that $3$-dimensional Cayley polytope equals
$$ {\rm gEDdegree}(X_A) \,\,=\,\, \delta_0 + \delta_1 + \delta_2 \,\,= \,\,
3 (c_1+c_2)  + {\rm max}(|c_1-c_2|, |c_3-c_4| ) . $$
Here, $\delta_0 = \delta_2 = c_1+c_2 = c_3 + c_4$,
and $\delta_1  =  \delta_0 + {\rm max}(|c_1-c_2|, |c_3-c_4| )$.
By (\ref{eq:volumepoly}), we find these formulas by measuring the
area of the planar polygon $\,\lambda P + \mu (-P)$.
\hfill $\diamondsuit$
\end{example}

\begin{proof}[Proof of Theorem \ref{thm:codim1}]
The map that attaches  tangent hyperplanes to
 smooth points of $X_A$ is a birational map from
 $X_A \subset \PP^{n-1}$ 
to  the conormal variety ${\rm Con}(X_A) \subset \PP^{n-1} \times \PP^{n-1}$.
It is equivariant with respect to the action of
the dense torus of $X_A$. Hence ${\rm Con}(X_A)$ is toric. We find its toric ideal using a procedure analogous to the transformation from  (\ref{eq:rank1a}) to (\ref{eq:rank1b}). Let $\mathcal{J}$ be the ideal given by the $2 \times 2$-minors of $\begin{pmatrix}
J(X_A) & y \end{pmatrix}^T $ where 
$y = (y_1,\ldots,y_n)$ and $J(X_A)$ is the gradient
vector of (\ref{eq:onebinomial}).
 This matrix is analogous to (\ref{eq:rank1a}). Let $I_A$ be the ideal of \eqref{eq:onebinomial}.
 
  The ideal defining ${\rm Con}(X_A)$ is $(I_A+\mathcal{J}): \langle J(X_A)\rangle^{\infty}$.
 This is a toric ideal. It can also be obtained by saturating  
the binomial ideal $ I_A+\mathcal{J}$ with respect to $x_1 \cdots x_n$
since the singular locus of $X_A$ lies in $\{x_1\cdots x_n = 0 \}$.
Among the generators of that toric ideal are the 
binomials $c_i x_j y_j - c_j x_i y_i$ as in (\ref{eq:rank1b}).
We take these for $j = i+1$ together with (\ref{eq:onebinomial})
and we write their exponents as the rows of the $n \times 2n$-matrix \eqref{eq:gale}.
   This matrix is the Gale dual of ${\rm Cay}(A,-A)$.
This proves the first two statements in Theorem \ref{thm:codim1}.
The  next conclusions about the ED degree and the polar degrees
of $X_A$ now follow from known results 
(cf.~\cite[Proposition 4.6]{MS})
about the relationship between
mixed volumes and triangulations of Cayley polytopes.
\end{proof}
Theorem~\ref{thm:codim1} identified the conormal
variety of a toric hypersurface as the toric variety
given by the Cayley polytope.
The ED degree is the volume of the Cayley polytope.
We now use
the general result in Theorem~\ref{thm:eins}
and~\ref{thm:zwei} to derive a formula for that volume.

\begin{theorem}
The $i^{th}$ polar degree of the toric hypersurface $X_A$ equals
\begin{equation}
\label{eq:codim1delta}
  \delta_i \,\,=\,\,\, \binom{n-1}{i+1} \cdot {\rm deg}(X_A)\, \,\,-\,
 \sum_{\tau\;:\;|\tau| = n-i-1} \!\!\!
{\rm min} \bigl(\,
\sum_{j \in \tau \cap \{1,\ldots,r\}} \!\!\!\!\!\! c_j \,\,\,,
\sum_{j \in \tau\cap \{r+1,\ldots,n\}}\!\!\!\!\!\!\! c_j \,\, \bigr).
\end{equation}
\end{theorem}

\begin{proof}
The $(n-2)$-dimensional polytope $P = {\rm conv}(A)$  is simplicial and
has $n$ vertices, provided $1 < r < n$. 
Following \cite[Section 6.5]{Zie}, the minimal non-faces of $P$ are
 $\{1,\ldots,r\}$ and $\{r{+}1,\ldots,n\}$. For $i \leq n-3$, we encode
each $i$-simplex in $\partial P$ by the index set
$\tau \subset \{1,2,\ldots,n\}$ of those columns $a_i$ that are not in that simplex.
 These $\tau$ satisfy
$|\tau| =  n-1-i$, and both $\tau^+ = \tau \cap \{1,\ldots,r\} $
and  $\tau^- = \tau \cap \{r{+}1,\ldots,n\} $ are non-empty.

By Corollary~\ref{cor:PL}, the polar degrees of $X_A$ are
linear functions on certain full-dimensional polyhedral cones in $\RR^n_{>0}$.
The lattice points  $(c_1,\ldots,c_n) $  with
  relatively prime coordinates in such a cone are Zariski dense.
   Every linear function on $\RR^n$ is determined by its values on
  a Zariski dense subset.
  Hence, in what follows, we may assume
that ${\rm gcd}(c_i,c_j) = 1$ for all~$i,j$.

Given this assumption, we claim  
that  ${\rm Vol}(\tau) = 1$ for every proper face $\tau$ of $P$.
Suppose  this does not hold.
Then ${\rm Vol}(\tau) > 1$ for some facet $\tau$,
say $\tau = \{r,n\}$ after relabeling.
This facet is the simplex with vertex set 
$\gamma = \{a_1,\ldots,a_{r-1},a_{r+1},\ldots,a_{n-1}\}$.
There exists $p \in \ZZ \gamma$ such that,
for some $i$, the lattice spanned by
$(\gamma \backslash \{a_i\}) \cup \{p\}$
has index $i_p \geq 2$ in $\ZZ \gamma$. We have
$$
\begin{matrix}
& c_r &=& {\rm Vol}\bigl(\,\gamma \cup \{a_n\}\, \bigr) &=& 
i_p \cdot {\rm Vol}\bigl(\,(\gamma \backslash \{a_i\}) \cup \{p,a_n\}\, \bigr) \\ 
{\rm and} &
c_n &=& {\rm Vol}\bigl(\,\gamma \cup \{a_r\}\, \bigr) &=& 
i_p \cdot {\rm Vol}\bigl(\,(\gamma \backslash \{a_i\}) \cup \{p,a_r\}\, \bigr).
\end{matrix}
$$
So, $i_p$ divides ${\rm gcd}(c_r,c_n)$,  a contradiction.
Hence ${\rm Vol}(\tau) = 1$ for every proper face $\tau$ of $P$.

For every face $\sigma$ of $P$ that contains  $\tau$, the subdiagram volume 
in Definition \ref{def:subdiagram} equals
\begin{equation}
\label{eq:SDV3}
\mu(\sigma/\tau) \,\,= \,\,
\begin{cases}\,
{\rm min} \bigl( \,\sum_{i \in \tau^+} c_i \,,\,
\sum_{j \in \tau^-} c_j \,\bigr) & {\rm if} \,\,\sigma = P ,\\ \qquad \qquad
1 & {\rm otherwise}.
\end{cases}
\end{equation}
With this, we can solve the recursion in Definition~\ref{def:Eu}.
 For a face $\alpha$ of $P$ let 
 $$
{\rm min}_A^{(r)}(\alpha)\,\,=\,\,{\rm min} ( \,\sum_{j \in \alpha \cap \{1,\ldots,r\}} \!\!\!\!\!\!\!c_j \,,\,
\sum_{j \in \alpha\cap \{r+1,\ldots,n\}}\!\!\!\!\!\!\! c_j \,). 
$$ 
From \eqref{eq:SDV3} and Definition \ref{def:Eu} we have 
\begin{equation}
{\rm Eu}(\tau) \,\,\,= \!\!\!\!\!\sum_{ {\beta \neq P {\rm \; s.t\;}\tau }\atop{ {\rm is \; a \; face\; of\;} \beta \; {\rm and} \atop \dim(\beta)=\dim(\tau)+1 }}\!\!\!\!\!\!\!(-1)^{n-\dim(\beta)-1} {\rm min}_A^{(r)}(\beta)
\,\,+\,\,(-1)^{n-\dim(\tau)-1} {\rm min}_A^{(r)}(\tau).
\end{equation}
This results in a formula for the CM volume of $\tau$, as an alternating sum
of expressions  ${\rm min} ( \,\sum_{j \in \sigma^+} c_j \,,\,
\sum_{j \in \sigma^-} c_j \,) $.
When we write the sum in (\ref{eq:Vjay}), and thereafter the sum in
(\ref{eq:polartoric}), a lot of regrouping and cancellation occurs.
The final result is the expression for $\delta_i$ in (\ref{eq:codim1delta}).
\end{proof}

\begin{corollary} \label{cor:codim1ED}
The generic Euclidean distance degree of the toric hypersurface $X_A$ equals
$$  {\rm gEDdegree}(X_A) \,\,=\,\,\, (2^{n-1}-1) \cdot {\rm deg}(X_A)\, \,\,-\,\,\,
 \sum_{\tau \subset \{1,\ldots,n\}}
   \!\!\!
{\rm min}\bigl(\,
\sum_{j \in \tau \cap \{1,\ldots,r\}} \!\!\!\!\!\! c_j\,\,\,,
\sum_{j \in \tau \cap \{r+1,\ldots,n\}}\!\!\!\!\!\!\! c_j\,\, \bigr).
$$
\end{corollary}

It is instructive to consider the case of surfaces in $\PP^3$
and to compare with Corollary~\ref{prop:surfaces}.

\begin{example} \rm
Let $n=4$ and $r=2$ and set $D = {\rm deg}(X_A)$.
The polar degrees are
$\,\,\delta_2 = D$,
$ \delta_1 = 3D - {\rm min}(c_1,c_3) - {\rm min}(c_1,c_4)
- {\rm min}(c_2,c_3) - {\rm min}(c_2,c_4) \,=\,
D + {\rm max}\bigl( |c_1-c_2|, |c_3-c_4| \bigr)$,
and
$\, \delta_0\, = \,3D - c_1-c_2-c_3-c_4\, =\,D$.
Their sum gives us the simple formula 
$$ {\rm gEDdegree}(X_A) \, = \, 3D + {\rm max}\bigl( |c_1-c_2|, |c_3-c_4| \bigr). $$
Another toric surface arises for $n=4$ and $r=1$. In that case,
$\delta_0 = \delta_2 = D$ and $\delta_1 = 2D$.
\hfill $\diamondsuit$
\end{example}

The results in this paper furnish exact formulas for the algebraic complexity 
of solving the optimization problems (\ref{eq:opt1}) and (\ref{eq:opt2}).
We close this section with a numerical example.

\begin{example} \rm
Given a  list $(u_1,u_2,u_3,u_4,u_5,u_6)$ of six real measurements,
we seek to find the best approximation by a 
real vector $(x_1,x_2,x_3,x_4,x_5,x_6)$ that satisfies the model
$$ x_1^{22} x_2^{23} x_3^{64} \,\,=\,\, x_4^{26} x_5^{14} x_6^{69}. $$
The general formula in \cite[Corollary~2.10]{DHOST} 
for hypersurfaces of degree $d= 109$ says that
$$ d  \cdot \bigl(1+(d-1)^1+(d-1)^2+(d-1)^3+(d-1)^4+(d-1)^5\bigr) \,\,=\,\, 1,616,535,525,241 $$
is a bound for the algebraic degree of our optimization problem.
Corollary~\ref{cor:codim1ED} shows that the true answer is much smaller:
${\rm gEDdegree}(X_A) = 1348$. {\em Numerical Algebraic Geometry} \cite{BHSW} allows us to
compute {\bf all} complex critical points, and hence {\bf all} local approximations.
\hfill $ \diamondsuit $
\end{example}

\section{Discriminants, Tropicalization and Hypersimplices}
\label{sec:five}

We computed the algebraic degree
of the optimization problem (\ref{eq:opt1}) when
the weight vector $\lambda$ and the data vector $u$ are generic.
This generic behavior fails when these vectors are zeros of
certain discriminants. In what follows we discuss those 
discriminants. Later in this section, we 
explore connections to tropical geometry:
building on \cite{DFS, DT}, we discuss the tropicalization 
of the conormal variety of a toric variety~$X_A$.
Thereafter, we conclude by returning to (\ref{eq:hypersimplexopt1}).

We begin by examining the genericity condition on the weight vector
$\lambda = (\lambda_1,\ldots,\lambda_n)$ that specifies the norm
$|| x ||_\lambda = (\sum_{i=1}^n \lambda_i x_i^2)^{1/2}$.
Following \cite{ottaviani2014exact}, we can define the  ED degree of 
the toric variety $X_A$ 
for any positive $\lambda$. However, it may be smaller than the generic one:
\begin{equation}
\label{eq:specialED}
 {\rm EDdegree}_\lambda(X_A) \,\, \leq \,\,  {\rm gEDdegree}(X_A).  
 \end{equation}
Such a drop occurred for $\lambda = (1,1,\ldots,1)$
in Example~\ref{ex:P1P1}, but not in Example \ref{ex:rnc}.
Similar instances are featured in
\cite[Example 2.7, Corollary 8.7]{DHOST} and
\cite[Examples 1.1,  Table 1, Proposition 4.1]{ottaviani2014exact}.
We now offer a characterization of the weights whose
ED degree is generic.

As before, we write $X_A^\vee$ for the $A$-discriminant, that is, the
projective variety dual to $X_A$. 
If the dual $X_A^\vee$ is a hypersurface in $\PP^{n-1}$ then $\Delta_A$ denotes its defining polynomial.
If ${\rm codim}(X_A^\vee) \geq 2$ then $\Delta_A = 1$.
Following \cite{GKZ} but ignoring exponents, we define the {\em principal $A$-determinant} $E_A$ to be the
product of the polynomials $\Delta_\alpha$ where $\alpha$ runs over all faces of~$A$.

\begin{proposition} \label{ref:eqholds}
Let $\lambda \in \RR^n_{> 0}$ be a weight vector such that the principal $A$-determinant
$E_A$ does not vanish at $\lambda$. Then equality holds in (\ref{eq:specialED}).
\end{proposition}

\begin{proof}
Theorem 5.4 in \cite{DHOST} states
that the ED degree of a variety $X \subset \PP^{n-1}$
agrees with the generic ED degree provided the
conormal variety ${\rm Con}(X)$ is disjoint from the
diagonal $\Delta(\PP^{n-1})$ in 
$\PP^{n-1} \times \PP^{n-1}$.
This refers to the usual Euclidean norm $|\!|\,\, |\!|_{\bf 1}$ on $\RR^n$.
We apply this to the scaled toric variety $X = \lambda^{1/2} X_A$ 
whose points are $\lambda^{1/2} x = (\lambda_1^{1/2} x_1 : \lambda_2^{1/2} x_2 : \cdots:
\lambda_n^{1/2} x_n)$ where $x = (x_1:x_2:\cdots:x_n)  $ runs over $X_A$.
If $x$ has non-zero coordinates then $x \in X_A$ means that
$x= (t^{a_1}: t^{a_2} : \cdots : t^{a_n})$ for some $t \in (\CC^*)^d$.
The ED problem for $X$ with respect to the norm
$|\!|\,\, |\!|_{\bf 1}$ is identical to the ED problem for $X_A$ with respect to
$|\!|\,\, |\!|_\lambda$.

Proposition \ref{ref:eqholds} 
claims that if the inequality in (\ref{eq:specialED}) is strict then $E_A(\lambda) = 0$.
Suppose that the inequality in (\ref{eq:specialED}) is strict.
By \cite[Theorem 5.4]{DHOST}, we know that  ${\rm Con}(X) \cap \Delta(\PP^{n-1}) $ is non-empty. 
Then there exists a point $x \in X_A$ such that
the hyperplane with normal vector $\lambda^{1/2} x $ is tangent to $X$ at the point $\lambda^{1/2} x $.
Let us first assume that $x$ has non-zero coordinates. Then
$ x = (t^{a_1}: t^{a_2} : \cdots : t^{a_n}) $ for some $t \in (\CC^*)^d$.
The tangency condition means that
the hypersurface defined by the Laurent polynomial $\sum_{i=1}^n \lambda_i t^{2 a_i}$ is singular
at the point $t \in (\CC^*)^d$. This implies that the hypersurface in the torus $(\CC^*)^d$ defined by
  the Laurent polynomial $\sum_{i=1}^n \lambda_i t^{a_i}$ is singular.
  We conclude that $\lambda$ lies in $X_A^\vee$, and hence $\Delta_A(\lambda) = 0$.
  
  Suppose now that some of the coordinates $x$ are zero. 
  Then the support of $x$ is a facial subset $\alpha$ of the columns of $A$.
  We now restrict to the torus orbit on $X_A$ given by that subset.
  The hyperplane with normal vector $\lambda^{1/2} x|_\alpha $ is tangent to $X_\alpha$
at the point $\lambda^{1/2} x|_\alpha $ in that orbit.
  By the same argument as in   the previous paragraph, we now find that
  $\Delta_\alpha(\lambda) = 0$.
  
Since the principal $A$-determinant $E_A$  is the product of
the $\alpha$-discriminants $\Delta_\alpha$ for all faces $\alpha$ of $A$,
we conclude that $E_A(\lambda) = 0$ holds whenever the inequality in
(\ref{eq:specialED}) is strict.
\end{proof}

\begin{example} \rm
Let $d=3$, $n=6$, and~$A =  \begin{pmatrix} 
2 & 1 & 1 & 0 & 0 & 0 \\
0 & 1 & 0 & 2 & 1 & 0 \\
0 & 0 & 1 & 0 & 1 & 2 \end{pmatrix} $.
Then $X_A$ is the Veronese surface in $\PP^5$, 
with  ${\rm gEDdegree}(X_A) = 13$, and
(\ref{eq:opt1}) is the problem of finding the best rank $1$ approximation to a given 
symmetric $3 \times 3$-matrix.  The principal $A$-determinant equals
$$ E_A(\lambda) \,\,= \,\,
{\rm det} \begin{small} \begin{pmatrix}
2 \lambda_1 & \lambda_2 & \lambda_3 \\
   \lambda_2 & 2 \lambda_4 & \lambda_5 \\
   \lambda_3 & \lambda_5 & 2 \lambda_6  
   \end{pmatrix} \end{small}
   \cdot    {\rm det}\begin{pmatrix}
2 \lambda_1 & \lambda_2 \\
   \lambda_2 & 2 \lambda_4 \end{pmatrix}
     \cdot    {\rm det}\begin{pmatrix}
2 \lambda_1 & \lambda_3 \\
   \lambda_3 & 2 \lambda_6 \end{pmatrix}
     \cdot    {\rm det}\begin{pmatrix}
2 \lambda_4 & \lambda_5 \\
   \lambda_5 & 2 \lambda_6 \end{pmatrix} \cdot \lambda_1 \lambda_4 \lambda_6.
   $$
If ${\rm EDdegree}_\lambda(X_A)$ drops below $13$ then
this product must be zero. 
We know from \cite[Example 3.2]{DHOST} that
${\rm EDdegree}_\lambda(X_A) $ drops down to $ 3$ when
$\lambda = (1,2,2,1,2,1)$. A computation reveals that
 ${\rm EDdegree}_\lambda(X_A) = 11$
when $\Delta_A(\lambda) \not=0$ but
 one of the $2 \times 2$-determinants vanishes.
\hfill $\diamondsuit$
\end{example}

\begin{remark} \rm
If all proper faces $\alpha$ of $A$ are affinely independent
then $E_A$ and $\Delta_A$ are equal up to a monomial factor,
so they have the same vanishing locus in $\RR^n_{>0}$.
If this holds and if the hypersurface  defined by
$\sum_{i=1}^n x_i = 0 $ inside $X_A$ is non-singular then
the usual Euclidean norm $|\!|\,\, |\!|_{\bf 1}$ exhibits the generic behavior, 
i.e.~${\rm EDdegree}_{\bf 1}(X_A)  = {\rm gEDdegree}(X_A) $.
This explains the generic behavior of $|\!|\,\, |\!|_{\bf 1}$ 
for rational normal curves in Example \ref{ex:rnc}, and for the next example.
\end{remark}

\begin{example} \rm
Consider the toric hypersurface (\ref{eq:onebinomial}). 
By \cite[\S 9.1]{GKZ}, its
$A$-discriminant equals
$$ \Delta_A \,\, \, = \,\,\, c_{r+1}^{c_{r+1}} \cdots c_n^{c_n} \cdot
\lambda_1^{c_1} \cdots \lambda_r^{c_r} \,\,-\,\,
(-1)^D  \cdot
c_1^{c_1} \cdots c_r^{c_r} \cdot
\lambda_{r+1}^{c_{r+1}} \cdots \lambda_n^{c_n} . $$
Hence $|\!|\,\, |\!|_{\bf 1}$ is always ED generic when
$D = {\rm deg}(X_A)$ is odd. If $D$ is even  then the hypothesis
$$ c_1^{c_1} \cdots c_r^{c_r} \,\, \not= \,\,
c_{r+1}^{c_{r+1}} \cdots c_n^{c_n}  $$
ensures that Corollary~\ref{cor:codim1ED}
counts  critical points correctly for the usual Euclidean norm.
\hfill $\diamondsuit$
\end{example}

Suppose now that $\lambda \in \RR^n_{> 0}$ with $E_A(\lambda) \not=0 $
has been fixed. The question arises which data vectors $u \in \RR^n$
exhibit the generic behavior. There are three possible types of
degeneracies:

\begin{itemize}
\item the {\em ED discriminant} \cite{DHOST} concerns collisions of
critical points in the smooth locus of $X_A$;
\item the {\em data singular locus} \cite[\S 2.1]{Hor} concerns critical points in
 the singular locus of $X_A$;
\item the {\em data isotropic locus} \cite[\S 2.2]{Hor} concerns critical points
that satisfy $\sum_{i=1}^n \lambda_i x_i^2 = 0$.
\end{itemize}
A careful study of all three for toric varieties $X_A$ would be worthwhile.
Generally none of these three loci are toric varieties themselves.
We offer some preliminary observations:
\begin{itemize}
\item Example 7.2 in \cite{DHOST} shows that the ED discriminant
is complicated and not toric even when $X_A$ has codimension $1$.
It would be interesting to compute the degree of the ED discriminant for
(\ref{eq:onebinomial}) and to compare it to 
Trifogli's formula in \cite[Theorem 7.3]{DHOST}. 
\item The data singular locus always contains the A-discriminant \cite[Theorem 1]{Hor}.
\item The data isotropic locus always contains the A-discriminant \cite[Theorem 2]{Hor}.
\end{itemize}

The Matsui-Takeuchi formula for the degree of the $A$-discriminant  given in
Theorem \ref{thm:zwei} is an alternating sum of CM volumes of faces of $P$.
A positive formula, as a sum of combinatorial numbers, was given 
independently by Dickenstein et al.~in \cite{DFS}. In fact, Theorem 1.2 in \cite{DFS}
expresses every initial monomial of $\Delta_A$ explicitly in a positive manner.
Such formulas are derived using  Tropical Geometry \cite{MS}. Their
advantage over \cite{MT} is that they furnish start systems for homotopy continuation 
in Numerical Algebraic Geometry \cite{BHSW}.

\smallskip

In what follows we assume familiarity with basics of tropical geometry,
especially on varieties given by monomials in linear forms
\cite[\S 5.5]{MS}. The {\em Horn uniformization} of the
$A$-discriminant \cite[\S 4]{DFS} lifts to
the following parametrization of the conormal variety of $X_A$.

\begin{proposition}
Let $A$ be an integer $d \times n$-matrix as above and $X_A$ its
projective toric variety in $\PP^{n-1}$.
The conormal variety ${\rm Con}(X_A)$ is the closure of the set of points
$(x,y)$ in $\PP^{n-1} \times \PP^{n-1}$,
where $x \in X_A$ and $x \cdot y \in {\rm kernel}(A)$.
Its tropicalization is the set of points $(u,  v)$  in $(\RR^{n}/\RR {\bf 1})^2$
where $u \in {\rm rowspace}(A)$
and $u+v$ is in the co-Bergman fan $\,\mathcal{B}^*(A)$.
\end{proposition}

\begin{proof}
The two statements are staightforward extensions of
\cite[Proposition 4.1]{DFS} and \cite[Corollary 4.3]{DFS} respectively,
obtained by keeping track of the tangent hyperplanes $H_\xi$  at $\xi \in X_A$.
\end{proof}

The tropical variety ${\rm trop}({\rm Con}(X_A))$
is a balanced fan of dimension $n-2$ in
$(\RR^{n}/\RR {\bf 1})^2$.
The description above was used by Dickenstein and Tabera \cite{DT}
to study singular hypersurfaces.

\begin{corollary}
The polar degree $\delta_i(X_A)$ is the number of
points in the intersection
$$ {\rm trop}({\rm Con}(X_A)) \,\cap \, (L_{n-2-i} \times M_{i}) \,\,
\,\, \subset \,\, (\RR^{n}/\RR {\bf 1}) \times  (\RR^{n}/\RR {\bf 1}),
$$
where $L_{n-2-i} $ is a tropical $(n-2-i)$-plane and $M_i$ is a tropical $i$-plane.
These planes can be chosen as in \cite[Corollary~3.6.16]{MS},
and the count is with multiplicities as in \cite[(3.6.5)]{MS}.
\end{corollary}

In analogy to \cite[Theorem 1.2]{DFS},
this corollary can be translated into an explicit positive formula
for the polar degrees and hence for the ED degree of $X_A$.
This should be useful for developing homotopy methods for solving
the critical equations, which can now be written~as
\begin{equation}
\label{eq:critical3} \qquad
 x +  y \,=\, u\, ,\,\,\,\,
x \in \tilde X_A\,\,\,\, {\rm and} \,\,\,\, x \cdot y \in {\rm kernel}(A) 
\qquad \qquad \hbox{for $\lambda = {\bf 1}$.}
\end{equation}
This formulation arises from 
\cite[Theorem 5.2]{DHOST},
where all varieties are regarded as affine cones.

\smallskip

We now return to the optimization problem (\ref{eq:hypersimplexopt1}).
Here $n = \binom{d}{k}$ and $A$ is the matrix
whose columns are the vectors in $\{0,1\}^d$ that have precisely
$k$ entries equal to $1$. The $(d-1)$-dimensional polytope
$P = {\rm conv}(A)$ is the {\em hypersimplex} $\Delta_{d,k}$.
The toric variety $X_A$ represents generic torus orbits
on the Grassmannian of $k$-dimensional linear subspaces in $\mathbb{C}^d$.
The degree of $X_A$ is the volume of $\Delta_{d,k}$. This
is known (by~\cite{Stanley}) to equal the {\em Eulerian number} $A(d-1,k-1) $.
In what follows we determine the CM volumes, polar degrees and gED degree
for the hyperpsimplex $\Delta_{d,k}$. 
Table \ref{tab:EDhypersimplex} offers a summary of all values for $d \leq 8$.
Here we may assume $2 \leq k \leq \lfloor d/2  \rfloor $ 
because the cases $(d,k)$ and $(d,d-k)$
are isomorphic.
\begin{table}[h]
$$
\begin{matrix}
d & k &  \hbox{Chern-Mather volumes} & \hbox{Polar degrees} & \hbox{gED degree} \\
4 & 2 & (12, 12, 8, 4) & (4, 12, 8, 4) & 28 \\
5 & 2 & (20, 30, 30, 25, 11) &  (5, 20, 40, 30, 11) &  106  \\
6 & 2 & (30, 60, 80, 90, 72, 26) & (6, 30, 80, 120, 84, 26) &  346  \\
6 & 3 & (60, 90, 120, 150, 132, 66) & ({\bf 96}, 300, 480, 480, 264, 66) & 1686 \\
7 & 2 & (42, 105, 175, 245, 273, 189, 57) & (7, 42, 140, 280, 336, 210, 57) & 1072 \\
7 & 3 & (105, 210, 350, 560, 714, 644, 302) &\! (315, \! 1302, \! 2940, \! 3920, \! 3192, \! 1470,\!  302) \!& 13441 \\ 
8 & 2 & \!(56, 168, 336, 560, 784, 784, 464, 120) \! & (8, 56, 224, 560, 896, 896, 496, 120) & 3256 \\
8 & 3 & (168, 420, 840, 1610, 2632, \ldots, 1191) & (848, 4256, 12096, 21280, 
\ldots, 1191) & 86647  \\
8 & 4 & (280, 560, 1120, 2240, \ldots, 2416) & 
(3816, 16016, 38976, 60480, \ldots 2416) & 
236104 \\
\end{matrix}
$$
\caption{\label{tab:EDhypersimplex} Computing the generic ED degree for the
toric variety of the hypersimplex $\Delta_{d,k}$}
\end{table}

A couple of observations are in place.
The last entry in the respective vectors is the
Eulerian number ${\rm Vol}(\Delta_{d,k}) = A(d-1,k-1)$.
The ED degree is the sum of the polar degrees.
The first polar degree $\delta_0$ is the
degree of the $A$-discriminant $\Delta_A$.
For $k=2$ this is simply the determinant of the symmetric
matrix with zero diagonal entries. For instance, for $d=4$,
\begin{equation}
\label{eq:missingdiagonal}
 \Delta_A (\lambda) \,\,  =\,\,
{\rm det} \begin{pmatrix}
 0 & \lambda_{12} & \lambda_{13} & \lambda_{14} \\
 \lambda_{12} & 0 & \lambda_{23} & \lambda_{24} \\
 \lambda_{13} & \lambda_{23} & 0 & \lambda_{34} \\
 \lambda_{14} & \lambda_{24} & \lambda_{34} &  0 \end{pmatrix} .
 \end{equation}
 A key point is that $\Delta_A(\lambda) \not=0 $
 when $\lambda = (1,\ldots,1)$. This ensures that the usual
 Euclidean metric is generic for  (\ref{eq:hypersimplexopt1}).
 There is no degree drop due to the weights $\lambda$ being~special.
 
   We close by presenting general formulas for the Chern-Mather volumes
   of hypersimplices:
   
\begin{proposition}
\label{prop:hypersimplex}   
The Chern-Mather volumes for the hypersimplex $\Delta_{d,k}$ are
$$
\begin{matrix}
       V_0 &=& \quad \binom{d}{k} \cdot {\rm min}(k,d-k) & \\
    V_\ell & = & \sum_{i=1}^{{\rm min}(k,\ell)} \! \binom{d}{\ell+1} \binom{d-\ell-1}{k-i} \cdot A(\ell,i-1) 
   & \qquad {\rm for} \,\,\, \ell = 1,\ldots,d-1.
   \end{matrix}
   $$
For $\ell = d-1$ this formula gives the Eulerian number $V_{d-1} = A(d-1,k-1) =
{\rm Vol}(\Delta_{d,k})$.
\end{proposition}

\begin{proof}
We apply the algorithm at the end of Section \ref{sec:three}
to the face poset of $\Delta_{d,k}$. Since 
every face of the hypersimplex is a hypersimplex,
it is convenient to proceed by induction. The base step 
is the subdiagram volume of a vertex of $\Delta_{d,k}$.
Each vertex has $(d-k)k$ neighbors.
These lie on a hyperplane in the ambient $(d-1)$-space.
Their convex hull is a product of simplices
$\Delta_{k-1} \times \Delta_{d-k-1}$. The normalized volume of such a product equals
$ \binom{d-2}{k-1}$. Hence the subdiagram volume of
 any vertex at $\Delta_{d,k}$ is $\binom{d-2}{k-1}$.
 The vertex figures of any positive-dimensional face at $\Delta_{d,k}$
is a simplex. In fact, the toric variety $X_{\Delta_{d,k}}$ has
isolated singularities. Hence  $\mu(\alpha/\beta) = 1$ 
for all subdiagram volumes at faces $\beta$ with ${\rm dim}(\beta) \geq 1$.
\end{proof}

From Proposition 4.7 one easily computes
the polar degrees  (\ref{eq:polartoric})
and the ED degree (\ref{eq:EDtoric}).
This solves an open problem, namely   to determine the
 degree of the $A$-discriminant for $k \geq 3$. This  was asked 
   for $d=6$ and $k=3$ in \cite[Problem 7]{HSYY}.
    Table~\ref{tab:EDhypersimplex} reveals that the answer is~$96$.

\bigskip \bigskip

\begin{small}

\noindent
{\bf Acknowledgements.}
We are grateful to Bernt Ivar N{\o}dland, Ragni Piene and Anna Seigal
for helpful comments on this paper.
We thank an anonymous referee for providing numerous comments that led
to improvements.
Martin Helmer was supported by an NSERC postdoctoral fellowship.
  Bernd Sturmfels was supported by the
  US National Science Foundation (DMS-1419018).
\end{small}

\bigskip

\begin{small}

\bibliographystyle{plain}

\end{small}

\bigskip

\noindent
\footnotesize {\bf Authors' address:}

\smallskip

\noindent 
Department of Mathematics, University of California, Berkeley, CA 94720, USA \\
{\tt martin.helmer@berkeley.edu}, \ {\tt bernd@berkeley.edu}

\end{document}